\documentclass[12]{amsart}
\usepackage{amsmath,amssymb,amsthm,color,enumerate,comment,centernot,enumitem,url,cite}
\usepackage{graphicx,relsize,bm}
\usepackage{mathtools}
\usepackage{array}

\makeatletter
\newcommand{\tpmod}[1]{{\@displayfalse\pmod{#1}}}
\makeatother

\newtheorem{thm}{Theorem}[section]

\newtheorem{cor}[thm]{Corollary}

\theoremstyle{remark}

\theoremstyle{definition}

\theoremstyle{THM}

\newcommand{\abs}[1]{\left|{#1}\right|}

\def\FF {{\mathcal F}}

\def\Z {{\mathbb Z}}

\def\Q {{\mathbb Q}}

\def\D {{\mathcal D}}

\def\F {{\mathbb F}}
\def\D {{\mathcal D}}
\def\Z {{\mathbb Z}}
\def\Q {{\mathbb Q}}

\def\Gal{{\mbox{{\rm{Gal}}}}}

\makeatletter
\@namedef{subjclassname@2020}{%
  \textup{2020} Mathematics Subject Classification}
\makeatother

\def\red#1 {\textcolor{red}{#1 }}
\def\blue#1 {\textcolor{blue}{#1 }}

\numberwithin{equation}{section}

\def\Z {{\mathbb Z}}

\newcommand{\Mod}[1]{\ (\mathrm{mod}\enspace #1)}
\newcommand{\mmod}[1]{\ \mathrm{mod}\enspace #1}
\begin{document}

\title[Monogenic sextic trinomials]{Monogenic sextic trinomials $x^6+Ax^3+B$ and their Galois groups}

\author{Joshua Harrington}
\address{Department of Mathematics, Cedar Crest College, Allentown, Pennsylvania, USA}
\email[Joshua Harrington]{Joshua.Harrington@cedarcrest.edu}

\author{Lenny Jones}
\address{Professor Emeritus, Department of Mathematics, Shippensburg University, Shippensburg, Pennsylvania 17257, USA}
\email[Lenny~Jones]{doctorlennyjones@gmail.com}

\date{\today}

\begin{abstract}
    Let $f(x)=x^6+Ax^3+B\in {\mathbb Z}[x]$, with $A\ne 0$, and suppose that $f(x)$ is irreducible over ${\mathbb Q}$. We define $f(x)$ to be {\em monogenic} if $\{1,\theta,\theta^2,\theta^3,\theta^4,\theta^{5}\}$ is a basis for the ring of integers of ${\mathbb Q}(\theta)$, where $f(\theta)=0$.
    
    For each possible Galois group $G$ of $f(x)$ over ${\mathbb Q}$, we use a theorem of Jakhar, Khanduja and Sangwan to give explicit descriptions of all monogenic trinomials $f(x)$ having Galois group $G$. We also investigate when these trinomials generate distinct sextic fields.
\end{abstract}


\maketitle
\section{Introduction}\label{Section:Intro}
 Let 
 \begin{equation}\label{Eq:f}
 f(x)=x^6+Ax^3+B\in {\mathbb Z}[x], \quad \mbox{where $AB\ne 0$}.
 \end{equation} If $f(x)$ is irreducible over ${\mathbb Q}$, then we define $f(x)$ to be {\em monogenic} if \[{\mathcal B}=\{1,\theta,\theta^2,\theta^3,\theta^4,\theta^{5}\}\] is a basis 
 for the ring of integers of ${\mathbb Q}(\theta)$, where $f(\theta)=0$. Such a basis ${\mathcal B}$ is sometimes referred to in the literature as a {\em power basis}. We say that two such monogenic sextic trinomials $f_1(x)\ne f_2(x)$ having Galois group $G$ are {\em distinct} if $K_1\ne K_2$, where $K_i=\Q(\theta_i)$ with $f_i(\theta_i)=0$.

 The purpose of this article is to characterize the monogenic sextic trinomials $f(x)$ described in \eqref{Eq:f} in terms of their Galois groups. That is, for each possible Galois group $\Gal(f)$ of $f(x)$ over ${\mathbb Q}$, we use a theorem of Jakhar, Khanduja and Sangwan \cite{JKS1} to derive conditions that allow us to provide an explicit description of all such monogenic sextic trinomials $f(x)$ having Galois group $\Gal(f)$. 
  We also investigate when these trinomials are distinct.
   
   We point out that similar research involving the monogenicity and/or Galois groups of trinomials of various degrees has been conducted by many authors  \cite{AJ,AL,BS,HJAA,HJBAMS,JonesQuarticsBAMS,JonesRam,JonesNYJM,JonesJAA,JonesAA,JonesRecipQuartics,JonesEvenSextics,JW,MNSU,S1,S2,Voutier}. In particular, the results in this article extend previous work of the second author on sextic trinomials \cite{JonesRam,JonesNYJM,JonesEvenSextics}.
 
 We make some notational remarks. Throughout this paper, for integers $m$ and $n$ with $m\ge 2$,  we let $n \mmod{m}$ denote the unique integer $z\in \{0,1,2,\ldots,m-1\}$ such that $n\equiv z \pmod{m}$. That is, $n \mmod{m}=z$. We let $C_n$ denote the cyclic group of order $n$, and $S_3$ denote the symmetric group of order $6$. At certain times, it will be convenient to let $\delta=A^2-4B$. 
 We also let $\nu_p(n)$ denote the $p$-adic valuation of the integer $n$ for any prime $p$.

Using the groups $C_n$ and $S_3$, the familiar names of the possible Galois groups $\Gal(f)$ \cite{AJ,Cav,HJMS} are provided in Table \ref{T1}. For the convenience of the reader, we also include the ``T"-notation \cite{BM} for these groups, which is given in Maple when computing the Galois group. 
    \begin{table}[h]
 \begin{center}
\begin{tabular}{c|ccccc}
  $\Gal(f)$ & $C_6$ & $S_3$ & $C_2\times S_3$ & $C_3\times S_3$ & $S_3\times S_3$\\ [.25em]   
 T-notation & 6T1 & 6T2 & 6T3 & 6T5 & 6T9\\ [.25em]
 \end{tabular}
\end{center}
\caption{Possible Galois groups for $f(x)=x^6+Ax^3+B$}
 \label{T1}
\end{table}

 Our main theorem is:
\begin{thm}\label{Thm:Main3}\text{}
Let $f(x)=x^6+Ax^3+B\in \Z[x]$  be irreducible over $\Q$, with $AB\ne 0$. Then
$f(x)$ is monogenic with $\Gal(f)\simeq$
\begin{enumerate}
  \item \label{M3:I1} $C_6$ if and only if $f(x)\in \FF_1:=\{x^6-x^3+1,\ x^6+x^3+1\}$;
  \item \label{M3:I1.5} $S_3$ never occurs;
 \item \label{M3:I2} $C_2\times S_3$ if and only if $f(x)\in \FF_i$ for some  $i\in \{2,3,4\}$, where
 \begin{enumerate}
   \item $\FF_2:=\{x^6-2x^3+2,\ x^6+2x^3+2\}$,
   \item $\FF_3:=\{x^6+Ax^3+1: A \mmod{9}\ne 0,\ A\ne \pm 1,$\\
    $\hspace*{1in} \mbox{with} \ A-2 \ \mbox{and}\ A+2 \ \mbox{squarefree}\}$,
   \item $\FF_4:=\left\{x^6+Ax^3-1: A \mmod{4}\ne 0, \ A \mmod{9}\not \in \{0,4,5\},\right.$\\
   $\hspace*{1in}\left. \mbox{with} \ (A^2+4)/\gcd(A^2+4,4) \ \mbox{squarefree} \right\}$;
    \end{enumerate}
  \item \label{M3:I3} $C_3\times S_3$ if and only if $f(x)\in \FF_5$, where
    \begin{align*}
  \FF_5:=&\{x^6+Ax^3+(A^2+3)/4: A \mmod{2}=1,\ A\ne \pm 1, \\
  &\hspace*{1.5in} \mbox{with} \ (A^2+3)/4 \ \mbox{squarefree}\};
  \end{align*}
   \item \label{M3:I4} $S_3\times S_3$ if and only if $f(x)$ is contained in one of $144$ infinite pairwise-disjoint $2$-parameter monogenic families. 
   \end{enumerate} Moreover, $\FF_i\cap \FF_j=\varnothing$ for $i\ne j$, and choosing all trinomials with positive coefficient on $x^3$ in each $\FF_i$  
  yields a complete set of distinct monogenic trinomials for the corresponding respective Galois groups. 
\end{thm}

\section{Basic Preliminaries}\label{Section:Prelim}
We begin with a theorem that provides conditions under which the possible Galois groups
of $f(x)=x^6+Ax^3+B$ from Table \ref{T1} can occur. 
\begin{thm}{\rm \cite{Cav,HJMS}}\label{Thm:HJCSextic1}
  Let $f(x)=x^6+Ax^3+B\in \Z[x]$  be irreducible over $\Q$, and let $R(x)=x^3-3Bx+AB$.
  Let $\delta=A^2-4B$. Then
  \[\Gal(f)\simeq \left\{\begin{array}{cl}
  C_6 & \mbox{if and only if $-3\delta$ is a square, $R(x)$ is irreducible and}\\
  & \mbox{$B=c^3$ for some $c\in \Z$,}\\[.5em]
  S_3 & \mbox{if and only if $-3\delta$ is a square, $R(x)$ is reducible and}\\
  & \mbox{$B\ne c^3$ for any $c\in \Z$,}\\[.5em]
  C_2\times S_3 & \mbox{if and only if $-3\delta$ is not a square and}\\
  & \mbox{either $R(x)$ is reducible or $B=c^3$ for some $c\in \Z$,}\\[.5em]
  C_3\times S_3 & \mbox{if and only if $-3\delta$ is a square, $R(x)$ is irreducible and}\\
  & \mbox{$B\ne c^3$ for any $c\in \Z$,}\\[.5em]
  S_3\times S_3 & \mbox{if and only if $-3\delta$ is not a square, $R(x)$ is irreducible and}\\
  & \mbox{$B\ne c^3$ for any $c\in \Z$.}
  \end{array}\right.\]
\end{thm}

The next theorem is due to the authors \cite[Theorem 3.1]{HJMS}.
\begin{thm}\label{Thm:m=3 HJMS}
  Let $g(x)=x^2+Ax+B\in \Z[x]$ be irreducible over $\Q$. Then
  \[f(x):=g(x^3)=x^6+Ax^3+B\] is reducible over $\Q$ if and only if $B=n^3$ and $A=m^3-3mn$ for some $m,n\in \Z$.
\end{thm}
The following immediate corollary of Theorem \ref{Thm:m=3 HJMS} will be useful in the proof of Theorem \ref{Thm:Main3}.
\begin{cor}\label{Cor:m=3 HJMS}
 Let $g(x)=x^2+Ax+B\in \Z[x]$ be irreducible over $\Q$, such that $B$ is squarefree with $\abs{B}\ne 1$. Then \[f(x):=g(x^3)=x^6+Ax^3+B\] is irreducible over $\Q$. 
\end{cor}

The formula for the discriminant of an arbitrary trinomial is given in the next theorem.
\begin{thm}
\label{Thm:Swan}
{\rm \cite{Swan}}
Let $f(x)=x^n+Ax^m+B\in \Z[x]$, where $0<m<n$, and let $d=\gcd(n,m)$. Then
\[
\Delta(f)=(-1)^{n(n-1)/2}B^{m-1}\left(n^{n/d}B^{(n-m)/d}-(-1)^{n/d}(n-m)^{(n-m)/d}m^{m/d}A^{n/d}\right)^d.
\]
\end{thm}

We now present some basic information concerning the monogenicity of a polynomial.
Suppose that $f(x)\in \Z[x]$ is monic and irreducible over $\Q$.
Let $K=\Q(\theta)$ with ring of integers $\Z_K$, where $f(\theta)=0$. Then, we have \cite{Cohen}
\begin{equation} \label{Eq:Dis-Dis}
\Delta(f)=\left[\Z_K:\Z[\theta]\right]^2\Delta(K),
\end{equation}
where $\Delta(f)$ and $\Delta(K)$ denote the discriminants over $\Q$, respectively, of $f(x)$ and the number field $K$.
Thus, we have the following theorem from \eqref{Eq:Dis-Dis}.
\begin{thm}\label{Thm:mono}
The polynomial $f(x)$ is monogenic if and only if
  $\Delta(f)=\Delta(K)$, or equivalently, $\Z_K=\Z[\theta]$.
\end{thm} We then have the following immediate corollary.
\begin{cor}\label{Cor:distinct}
  Let $f_1(x)\ne f_2(x)$ be two monogenic polynomials such that $\deg(f_1)=\deg(f_2)$. Let $K_i=\Q(\theta_i)$, where $f_i(\theta_i)=0$. If $\Delta(f_1)\ne \Delta(f_2)$, then 
  $f_1(x)$ and $f_2(x)$ are distinct.
\end{cor}

The next theorem, due to Jakhar, Khanduja and Sangwan \cite{JKS1}, gives necessary and sufficient conditions for an irreducible trinomial $f(x)=x^{n}+Ax^m+B\in \Z[x]$, where $m\ge 1$ is a proper divisor of $n$, to be monogenic. Although the same authors have also proven a version of this theorem that addresses arbitrary irreducible trinomials \cite{JKS2}, we do not require that more general version here. 
\begin{thm}\label{Thm:JKS1}
Let $n\ge 2$ be an integer.
Let $K=\Q(\theta)$ be an algebraic number field with $\theta\in \Z_K$, the ring of integers of $K$, having minimal polynomial $f(x)=x^{n}+Ax^m+B$ over $\Q$, where $m\ge 1$ is a proper divisor of $n$. Let $t=n/m$. A prime factor $p$ of $\Delta(f)$ does not divide $\left[\Z_K:\Z[\theta]\right]$ if and only if $p$ satisfies one of the following conditions:
\begin{enumerate}
  \item \label{JKS:C1Main} when $p\mid A$ and $p\mid B$, then $p^2\nmid B$;
  \item \label{JKS:C2Main} when $p\mid A$ and $p\nmid B$, then
  \[\mbox{either } \quad p\mid a_2 \mbox{ and } p\nmid b_1 \quad \mbox{ or } \quad p\nmid a_2\left(a_2^tB+\left(-b_1\right)^t\right),\]
  where $a_2=A/p$ and $b_1=\frac{B+(-B)^{p^j}}{p}$ with $p^j\mid\mid tm$;
  \item \label{JKS:C3Main} when $p\nmid A$ and $p\mid B$, then
  \[\mbox{either } \quad p\mid a_1 \mbox{ and } p\nmid b_2 \quad \mbox{ or } \quad p\nmid a_1b_2^{m-1}\left(Aa_1^{t-1}+\left(-b_2\right)^{t-1}\right),\]
  where $a_1=\frac{A+(-A)^{p^l}}{p}$ with $p^l\mid\mid (t-1)m$, and $b_2=B/p$;
  \item \label{JKS:C4} when $p\nmid AB$ and $p\mid m$ with $n=s^{\prime}p^k$, $m=sp^k$, $p\nmid \gcd\left(s^{\prime},s\right)$, then 
   \begin{equation*}
     H_1(x):=x^{s^{\prime}}+Ax^s+B \quad \mbox{and}\quad H_2(x):=\dfrac{Ax^{sp^k}+B+\left(-Ax^s-B\right)^{p^k}}{p}
   \end{equation*}
   are coprime modulo $p$;
         \item \label{JKS:C5Main} when $p\nmid ABm$, then $p^2\nmid \left(t^tB^{t-1}-(-1)^{t}(t-1)^{t-1}A^t\right)$.
   \end{enumerate}
\end{thm}

\section{The Proof of Theorem \ref{Thm:Main3}}\label{Section:Main3Proof}

Before we begin the proof Theorem \ref{Thm:Main3}, we present an adaptation of Theorem \ref{Thm:JKS1} to the specific trinomial $f(x)=x^6+Ax^3+B$.

\begin{thm}\label{Thm:m=3}
Suppose that $f(x)=x^6+Ax^3+B\in \Z[x]$ is irreducible over $\Q$. Note that
\begin{equation}\label{Eq:Delf}
\Delta(f)=3^6B^2(A^2-4B)^3 
\end{equation} by Theorem \ref{Thm:Swan}. Let $K=\Q(\theta)$, where $f(\theta)=0$, and let $\Z_K$ denote the ring of integers of $K$. A prime factor $p$ of $\Delta(f)$ does not divide $\left[\Z_K:\Z[\theta]\right]$ if and only if $p$ satisfies one of the following conditions:
\begin{enumerate}
  \item \label{m=3:C1} when $p\mid A$ and $p\mid B$, then $p^2\nmid B$;
  \item \label{m=3:C2} when $p\mid A$ and $p\nmid B$, then $p\in \{2,3\}$ with\\ 
    \[(A \mmod{4},\ B \mmod{4})\in \{(0,1), (2,3)\} \quad \mbox{if $p=2$,}\]
    \begin{align*}
  \mbox{and}\quad (A \mmod{9},\ B \mmod{9}) & \in \{(0,2), (0,4), (0,5), (0,7),\\
  & \qquad (3,1), (3,4), (3,7), (3,8),\\
  & \qquad (6,1), (6,4), (6,7), (6,8)\} \quad \mbox{if $p=3$};
  \end{align*}
  \item \label{m=3:C3} when $p\nmid A$ and $p\mid B$, then $p^2\nmid B$ and
    \begin{align*}
  (A \mmod{9},\ B \mmod{9}) & \in \{(1,3), (1,6), (2,3), (4,6),\\
  & \qquad (5,6), (7,3), (8,3), (8,6) \} \quad \mbox{if $p=3$};
  \end{align*}
  \item \label{m=3:C4} when $3\nmid AB$, then 
    \begin{align*}
  (A \mmod{9},\ B \mmod{9}) & \in \{(1, 1), (1, 2), (1, 4), (1, 5), (1, 8), (2, 2), (2, 4),\\
   & \qquad (2, 5), (2, 7), (2, 8), (4, 1), (4, 2), (4, 5), (4, 7),\\
   & \qquad (5, 1), (5, 2), (5, 5), (5, 7), (7, 2), (7, 4), (7, 5),\\
   & \qquad (7, 7), (7, 8), (8, 1), (8, 2), (8, 4), (8, 5), (8, 8)\};
  \end{align*}
  \item \label{m=3:C5} when $p\nmid 3AB$, then $p^2\nmid (A^2-4B)$. 
   \end{enumerate}
\end{thm}
\begin{proof}
 Since the methods used for the adaptation of Theorem \ref{Thm:JKS1} to the specific trinomial $f(x)=x^6+Ax^3+B$ are straightforward and mainly computational, we provide details only for condition \eqref{m=3:C2}. Suppose that $p\mid A$ and $p\nmid B$. Then, from condition \eqref{JKS:C2Main} of Theorem \ref{Thm:JKS1}, we have that
 $n=6$, $m=3$, $t=2$ and $p^j\mid \mid 6$. Hence, $p\in \{2,3\}$ with $a_2=A/p$ and
 \[b_1=\left\{\begin{array}{cl}
   \frac{B+B^2}{2} & \mbox{if $p=2$}\\[.5em]
   \frac{B-B^3}{3} & \mbox{if $p=3$.}
 \end{array}\right.\] Then, it is easy to calculate the congruence classes of $A$ and $B$ modulo $p^2$ for which either $p\mid a_2$ and $p\nmid b_1$, or $p\nmid a_2(a_2^2B+b_1^2)$.
\end{proof}
 We observe that further refinements of Theorem \ref{Thm:m=3} are possible. In particular, if $f(x)$ is monogenic, then $B$ must be squarefree from conditions \eqref{m=3:C1} and \eqref{m=3:C3}.
 Consequently,
 \begin{equation}\label{Eq:Gamma}
 \Gamma:=(A^2-4B)/(2^{\nu_2(A^2-4B)}3^{\nu_3(A^2-4B)})
  \end{equation} must also be squarefree from condition \eqref{m=3:C5} of Theorem \ref{Thm:m=3}. 
  These observations will be useful in establishing the converse of each part of Theorem \ref{Thm:m=3}. More precisely,  for a given Galois group $G$ and a monogenic trinomial $f(x)=x^6+Ax^3+B$ with $\Gal(f)\simeq G$, the fact that both $B$ and $\Gamma$ are squarefree will be especially helpful in showing that $f(x)\in \FF_i$ for the appropriate value of $i$. 

\begin{proof}[Proof of Theorem \ref{Thm:Main3}]

Considering, one at a time, each of the possible Galois groups, 
\[C_6,\ S_3, \ C_2\times S_3, \ C_3\times S_3,\] we use \eqref{Eq:Delf}, Theorem \ref{Thm:HJCSextic1} and Theorem \ref{Thm:m=3} to show that each $f(x)\in \FF_i$ is monogenic with the prescribed Galois group. We then establish the claim concerning when the trinomials in $\FF_i$ are distinct, and show that each $\FF_i$, with $i\ge 3$, contains infinitely many distinct trinomials. Finally, we prove the converse for each of these groups. For the special case $\Gal(f)\simeq S_3 \times S_3$, we use Theorem \ref{Thm:m=3} to illustrate how to construct infinite 2-parameter families to ``capture" all monogenic trinomials $f(x)$. 

 We recall that
\begin{equation}\label{Eq:R}
R(x)=x^3-3Bx+AB.
\end{equation}
\subsection*{{\bf The Case} $\mathbf{C_6}$}
\subsubsection*{$\mathbf{\FF_1}$}
Let $f^-(x)=x^6-x^3+1$ and $f^+(x)=x^6+x^3+1$. Straightforward calculations (using Maple, \eqref{Eq:Delf}, Theorem \ref{Thm:HJCSextic1} and Theorem \ref{Thm:m=3})
confirm that $f^-(x)$ and $f^+(x)$ are irreducible and monogenic with
\[\Delta(f^-)=\Delta(f^+)=-3^9 \quad \mbox{and} \quad \Gal(f^-)=\Gal(f^+)\simeq C_6.\]
 Furthermore, if $f^+(\theta)=0$, then since $f^-(-\theta)=0$, it follows that $f^-(x)$ and $f^+(x)$ are not distinct. 
 
 For the converse, suppose that $f(x)=x^6+Ax^3+B$ is monogenic with $\Gal(f)\simeq C_6$. Then $B=c^3$ for some nonzero integer $c$, by Theorem \ref{Thm:HJCSextic1}. But $B$ is squarefree, since $f(x)$ is monogenic. Hence, $B=\pm 1$. Also, by Theorem \ref{Thm:HJCSextic1}, we have that 
 $-3(A^2-4B)$ is a square, which is impossible if $B=-1$. Thus, $B=1$ so that 
 \[-3(A^2-4B)=-3(A^2-4)=12-3A^2\] is a square. Note that $12-3A^2\ne 0$ since $f(x)$ is irreducible. Therefore, $A=\pm 1$, which completes the proof of item \eqref{M3:I1} of the theorem. 

\subsection*{{\bf The Case} $\mathbf{S_3}$} Suppose that $f(x)=x^6+Ax^3+B\in \Z[x]$  is irreducible over $\Q$ with $\Gal(f)\simeq S_3$. Observe that if $2\nmid AB$, then
\begin{equation}\label{Eq:R mod 2}
R(x)\equiv x^3+x+1 \pmod{2}
\end{equation} is irreducible in $\F_2[x]$, and consequently, $R(x)$ is irreducible over $\Q$, which contradicts Theorem \ref{Thm:HJCSextic1}. Therefore, $2\mid AB$. Note that if $2\nmid A$ and $2\mid B$, then
\[-3(A^2-4B)\equiv 5 \pmod{8},\] which contradicts the fact that $-3(A^2-4B)$ is a square from Theorem \ref{Thm:HJCSextic1}.
To rule out the final two possibilities
\[(A\mmod{2},\ B\mmod{2})\in \{(0,0),\ (0,1)\},\]
we assume that $f(x)$ is monogenic, and proceed by way of contradiction invoking Theorem \ref{Thm:m=3}.

If $2\mid A$ and $2\mid B$, then $B \mmod{4}=2$ by condition \eqref{m=3:C1} of Theorem \ref{Thm:m=3}, and $4\mid (A^2-4B)$. Thus, 
\[\frac{-3(A^2-4B)}{4}=-3\left((A/2)^2-B\right)\equiv \left\{\begin{array}{cl}
  2 \pmod{4} & \mbox{if $4\mid A$,}\\
  3 \pmod{4} &  \mbox{if $4\nmid A$,}
\end{array}\right.\] which contradicts the fact that $-3(A^2-4B)$ is a square from Theorem \ref{Thm:HJCSextic1}.

The final situation to consider is when $2\mid A$ and $2\nmid B$.  Since $f(x)$ is monogenic and $4\mid (A^2-4B)$, we have from condition \eqref{m=3:C2} of Theorem \ref{Thm:m=3} with $p=2$ that \[(A \mmod{4},\ B \mmod{4})\in \{(0,1), (2,3)\}.\] But it is then easy to verify that
\[\frac{-3(A^2-4B)}{4}\equiv \left\{\begin{array}{cl}
  3 \pmod{4} & \mbox{if $(A \mmod{4},\ B \mmod{4})=(0,1)$},\\
  2 \pmod{4} & \mbox{if $(A \mmod{4},\ B \mmod{4})=(2,3)$},\\
\end{array}\right.\] which again contradicts the fact that $-3(A^2-4B)$ is a square from Theorem \ref{Thm:HJCSextic1}. This final contradiction completes the proof of item \eqref{M3:I1.5}.

\subsection*{{\bf The Case} $\mathbf{C_2 \times S_3}$}
\subsubsection*{$\mathbf{\FF_2}$}  Since the arguments here for the monogenicity of the two elements of $\FF_2$, and the fact that they are not distinct,  are similar to the case $\FF_1$, we omit the details.

 \subsubsection*{$\mathbf{\FF_3}$} Let $f(x)=x^6+Ax^3+1$ such that
 \[A\mmod{9}\ne 0,\quad  A\ne \pm 1,\quad \mbox{with}\ A-2 \ \mbox{and} \ A+2 \ \mbox{squarefree.}\] The restriction $A\mmod{9}\ne 0$ is to avoid the reducibility of $f(x)$ when $A=0$, and the nonmonogenicity of $f(x)$ otherwise. The restriction that $A\ne \pm 1$ is to avoid overlap with $\FF_1$.
 The restriction that $A-2$ and $A+2$ be squarefree is required for the irreducibility and monogenicity of $f(x)$, as we shall see below.

 We show first that $f(x)$ is irreducible over $\Q$. Let $h(x)=x^2+Ax+1$. It is easy to see that the only solutions to the Diophantine equation $A^2-4=y^2$ are $(A,y)\in \{(-2,0), (2,0)\}$. Hence, $h(x)$ is irreducible over $\Q$ since $A-2$ and $A+2$ are squarefree. Thus, if $f(x)=h(x^3)$ is reducible over $\Q$, we have by Theorem \ref{Thm:m=3 HJMS} that $A=m^3-3m$ for some $m\in \Z$. But then
 \[A-2=(m-2)(m+1)^2 \quad \mbox{and} \quad A+2=(m+2)(m-1)^2,\] which is impossible since $A-2$ and $A+2$ are squarefree, and $A\ne 0$. Therefore, $f(x)$ is irreducible over $\Q$.

 Observe that
 \begin{equation}\label{Delf F3}
 \Delta(f)=3^6(A-2)^3(A+2)^3
 \end{equation} from \eqref{Eq:Delf}. For the monogenicity of $f(x)$, we first consider the prime divisor $p=3$ of $\Delta(f)$. If $3\mid A$, then
 \[(A \mmod{9},\ B\mmod{9})\in \{(3,1),(6,1)\},\] so that condition \eqref{m=3:C2} of Theorem \ref{Thm:m=3} is satisfied. If $3\nmid A$, then $A\not \equiv \pm 2\Mod{9}$ since $A\mp 2$ is squarefree. Hence,
 \[(A \mmod{9},\ B \mmod{9})\in \{(1,1),(4,1),(5,1),(8,1)\},\] so that condition \eqref{m=3:C3} of Theorem \ref{Thm:m=3} is satisfied.

 Next, suppose that $p\mid (A-2)$ with $p\ne 3$. If $p\mid A$, then $p=2$. It follows that $A \mmod{4}=0$ since $A-2$ and $A+2$ are squarefree. Hence, $(A\mmod{4},\ B\mmod{4})=(0,1)$ so that condition \eqref{m=3:C2} of Theorem \ref{Thm:m=3} is satisfied. If $p\nmid A$, then $p\ne 2$ and $p\nmid (A+2)$ so that $p^2\nmid (A^2-4)$ since
  $A-2$ and $A+2$ are squarefree. A similar argument shows that $p^2\nmid (A^2-4)$ if $p\mid (A+2)$ with $p\ne 3$. Hence, condition \eqref{m=3:C5} of Theorem \ref{Thm:m=3} is satisfied, and we conclude from Theorem \ref{Thm:m=3} that $f(x)$ is monogenic.

    Note that $A^2-4>0$ from the restrictions on $A$. Thus, $-3(A^2-4B)$ is not a square. Since $B=c^3$ with $c=1$, we conclude from Theorem \ref{Thm:HJCSextic1} that $\Gal(f)\simeq C_2 \times S_3$.

  Suppose that
  \[f_1(x)=x^6+A_1x^3+1\in \FF_3 \quad \mbox{and} \quad f_2(x)=x^6+A_2x^3+1\in \FF_3,\] such that $f_1(x)\ne f_2(x)$ and $f_i(\theta_i)=0$. Since $A_1\ne A_2$ and $A_i^2-4>0$, we see from \eqref{Delf F3} that if $\Delta(f_1)=\Delta(f_2)$, then
  \[(A_1-A_2)(A_1+A_2)\left((A_1^2-4)^2+(A_1^2-4)(A_2^2-4)+(A_2^2-4)^2\right)=0,\] which implies that $A_1=-A_2$. Consequently, trinomials in $\FF_3$ with positive coefficient on $x^3$  are distinct by Corollary \ref{Cor:distinct}.

   Since $9\nmid A$, suppose that $A=9k+r$ for some fixed integer $r$ with $9\nmid r$.  Then, there exist infinitely many integers $k>0$ such that
  \[(A-2)(A+2)=(9k-r-2)(9k+r+2)\] is squarefree \cite{BB}. Hence, we conclude that there exist infinitely many distinct trinomials in $\FF_3$.

  \subsubsection*{$\mathbf{\FF_4}$} Let $f(x)=x^6+Ax^3-1$ such that
 \[A\mmod{4}\ne 0,\quad  A \mmod{9}\not \in \{0,4,5\},\quad \mbox{with}\ (A^2+4)/\gcd(A^2+4,4) \ \mbox{squarefree.}\]
 The argument for the irreducibility of $f(x)$ here is similar to the situation of $\FF_3$, and we omit the details.

 From \eqref{Eq:Delf}, we have
 \begin{equation*}\label{Delf F4}
 \Delta(f)=3^6(A^2+4)^3.
 \end{equation*} To establish monogenicity, we proceed as before, starting with $p=3$. If $3\mid A$, then
 \[(A \mmod{9},\ B\mmod{9})\in \{(3,8),(6,8)\},\] so that condition \eqref{m=3:C2} of Theorem \ref{Thm:m=3} is satisfied. If $3\nmid A$, then
 \[(A \mmod{9},\ B\mmod{9})\in \{(1,8),(2,8),(7,8),(8,8)\},\] so that condition \eqref{m=3:C4} of Theorem \ref{Thm:m=3} is satisfied.

 Next, suppose that $p\mid (A^2+4)$ with $p\ne 3$. If $p=2$, then $2\mid A$ so that $(A \mmod{4},\ B\mmod{4})=(2,3)$ since $4\nmid A$. Hence, condition \eqref{m=3:C2} of Theorem \ref{Thm:m=3} is satisfied. Assume then that $p\ne 2$. Note that $p\nmid A$. Thus, $p^2\nmid (A^2+4)$ since $(A^2+4)/\gcd(A^2+4,4)=A^2+4$ is squarefree. Therefore, condition \eqref{m=3:C5} of Theorem \ref{Thm:m=3} is satisfied, and $f(x)$ is monogenic.

 Arguments similar to the situation of $\FF_3$ verify here that $\Gal(f)\simeq C_2 \times S_3$ and that trinomials in $\FF_4$ with positive coefficient on $x^3$ are distinct. Moreover, if $A=36k+1$, then there exist infinitely many positive integers $k$ such that
 \[\frac{A^2+4}{\gcd(A^2+4,4)}=\frac{(36k+1)^2+4}{\gcd((36k+1)^2+4,4)}=1296k^2+72k+5\] is squarefree \cite{Nagel}, which implies that $\FF_4$ contains infinitely many distinct trinomials.

 We turn now to the converse. That is, we assume that $f(x)=x^6+Ax^3+B$ is monogenic with $\Gal(f)\simeq C_2 \times S_3$, and we show that $f(x)\in \FF_2\cup \FF_3\cup \FF_4$. Since $f(x)$ is monogenic, we recall that $B$ and $\Gamma$ (from \eqref{Eq:Gamma}) are squarefree from conditions \eqref{m=3:C1}, \eqref{m=3:C3} and \eqref{m=3:C5} of Theorem \ref{Thm:m=3}.

 Suppose first that $\abs{B}\ge 2$. Then, since $B$ is squarefree, it follows that $B\ne c^3$ for any $c\in \Z$. Hence, $R(x)$ must be reducible by Theorem \ref{Thm:HJCSextic1} since $\Gal(f)\simeq C_2 \times S_3$. Consequently, $2\mid AB$ since otherwise $R(x)$ is irreducible over $\Q$, as in the case of $S_3$ in \eqref{Eq:R mod 2}.  Furthermore, if $p$ is a prime such that $p\mid B$ and $p\nmid A$, then $R(x)$ is $p$-Eisenstein since $B$ is squarefree, which implies the contradiction that $R(x)$ is irreducible over $\Q$. We deduce therefore that all primes dividing $B$ must also divide $A$. Hence, since $B$ is squarefree, it follows that
  \begin{equation}\label{Eq:B,2|A}
  B\mid A \quad \mbox{and} \quad 2\mid A.
   \end{equation} 

 Since $R(x)$ is reducible over $\Q$, we have that
 \begin{equation}\label{Eq:R reducible}
  R(x)=(x-r)(x^2+rx+r^2-3B),
 \end{equation}  for some $r\in \Z$.
 Calculating $\Delta(R)$ in two ways, using Theorem \ref{Thm:Swan} for \eqref{Eq:R} and \cite{Janson} for \eqref{Eq:R reducible}, yields
  \begin{equation}\label{Eq:Del R}
  \Delta(R)=-27(A^2-4B)B^2=-27(r^2-4B)(r^2-B)^2.
   \end{equation}  Since $B$ is squarefree, it follows from \eqref{Eq:Del R} that $B\mid r$, so that
   \begin{equation}\label{Eq:From Del R}
   A^2-4B=(r^2-4B)\left(\dfrac{r^2}{B}-1\right)^2.
   \end{equation} Hence,
   \[\Gamma=\frac{A^2-4B}{2^{\nu_2(A^2-4B)}3^{\nu_3(A^2-4B)}}
   =\frac{(r^2-4B)\left(\dfrac{r^2}{B}-1\right)^2}{2^{\nu_2(A^2-4B)}3^{\nu_3(A^2-4B)}}\] is squarefree, and consequently,
   \begin{equation}\label{Eq:Exam}
   \abs{\dfrac{r^2}{B}-1}=2^a3^b
    \end{equation} for some nonnegative integers $a$ and $b$.

    We claim that $b=0$ in \eqref{Eq:Exam}. Suppose, by way of contradiction, that $b \geq 1$. Then,
    \begin{equation}\label{Eq:divisible by 3}
    r^2/B-1\equiv 0 \pmod{3},
     \end{equation} so that $9\mid(r^2/B-1)^2$, which implies that
     \begin{equation}\label{Eq:divisible by 9}
     9\mid (A^2-4B)
     \end{equation} from \eqref{Eq:From Del R}. Hence, since $B$ is squarefree and $B\mid r$, we have that  $3\mid (r^2/B)$, which contradicts \eqref{Eq:divisible by 3}. Thus, $B \equiv r^2\equiv 1 \pmod{3}$, so that $3\nmid A$ from \eqref{Eq:divisible by 9}. Consequently, we have
     \[A\mmod {9}\in \{1,2,4,5,7,8\} \quad \mbox{and}\quad B \mmod{9} \in \{1, 4, 7\}.\] Thus, in light of  \eqref{Eq:divisible by 9}, straightforward calculations reveal that
\[(A\mmod{9},\ B\mmod{9}) \in \{(1,7), (2,1), (4,4), (5,4), (7,1), (8,7)\},\] which contradicts condition \eqref{m=3:C4} of Theorem \ref{Thm:m=3} since $f(x)$ is monogenic. Thus, the claim that $b=0$ is established and
\begin{equation}\label{Eq:Exam1}
   \abs{\dfrac{r^2}{B}-1}=2^a
    \end{equation} for some nonnegative integer $a$.

We claim next that $a\le 1$ in \eqref{Eq:Exam1}. Assume, by way of contradiction, that $a\ge 2$. Then $16\mid (r^2/B-1)^2$ from \eqref{Eq:Exam1}, so that 
\begin{equation}\label{Eq:divisible by 16}
16\mid (A^2-4B)
\end{equation} from \eqref{Eq:From Del R}. If $2\mid B$, then $r^2/B-1\equiv 1\pmod{2}$ since $B$ is squarefree, which contradicts \eqref{Eq:Exam1}. Hence,  $2\nmid B$.  Thus, since $2\mid A$ from \eqref{Eq:B,2|A}, we see that
\[A\mmod {16}\in \{0,2,4,6,8,10,12,14\} \quad \mbox{and}\quad B \mmod{16} \in \{1, 3,5,7,9,11,13,15\}.\] Checking these values with the added restriction of \eqref{Eq:divisible by 16} produces 
\begin{multline*}\label{Eq:Pairs mod 16}
(A\mmod{16},\ B\mmod{16}) \in S=\{(2, 1), (2, 5), (2, 9), (2, 13), (6, 1), (6, 5), (6, 9), \\
 (6, 13),(10, 1), (10, 5), (10, 9), (10, 13), (14, 1), (14, 5), (14, 9), (14, 13)\}.
 \end{multline*} Observe that the reduction modulo 4 of every pair in $S$ results in the same pair $(A \mmod{4},B\mmod{4})=(2,1)$, which contradicts condition \eqref{m=3:C2} of Theorem \ref{Thm:m=3} since $f(x)$ is monogenic. Thus, we have established the claim that $a\le 1$ in \eqref{Eq:Exam1}.

 Suppose then that $a=0$ in \eqref{Eq:Exam1}. If $r^2/B-1<0$, then $r=0$, which yields the contradiction that $AB=0$ from \eqref{Eq:R} since $R(r)=0$. If $r^2/B-1>0$, then, since $B$ is squarefree, we get the two solutions
    \[(A,B,r)\in \{(2,2,2), (-2,2,-2)\}\] from \eqref{Eq:R} since $R(r)=0$, which correspond precisely to the elements of $\FF_2$.

 Suppose next  that $a=1$ in \eqref{Eq:Exam1}. Then 
 \[\frac{r^2}{B}-1=\pm 2,\] from which we deduce that  
  \begin{equation}\label{Eq:B poss}
  (r,B)\in \{(-1,-1),\ (1,-1),\ (-3,3), \ (3,3) \}.
    \end{equation} The first two pairs in \eqref{Eq:B poss} are impossible since $\abs{B}\ge 2$, while the last two pairs both produce the contradiction that $A=0$ by \eqref{Eq:R} since $R(r)=0$. 
        Therefore, we may assume that $\abs{B}=1$. 
        
        We first consider $B=1$, so that $f(x)=x^6+Ax^3+1$. Note that if $\abs{A}=1$, then \[-3\delta=-3(A^2-4B)=9,\] which contradicts the fact that $-3\delta$ is not a square from Theorem \ref{Thm:HJCSextic1}. Hence, $\abs{A}\ne 1$. If $9\mid A$, then 
        \[(A\mmod{9}, \ B\mmod{9})=(0,1),\] which contradicts condition \eqref{m=3:C2} of Theorem \ref{Thm:m=3} since $f(x)$ is monogenic. Thus, $9\nmid A$. 
        Next, we prove that $A-2$ is squarefree. Assume, by way of contradiction, that $A-2$ is not squarefree, and let $p$ be a prime divisor of $A-2$ such that $p^2\mid (A-2)$. Then $p^2\mid (A^2-4)$, and $p\mid 3A$ by condition \eqref{m=3:C5} of Theorem \ref{Thm:m=3}. Suppose that $p\mid A$. Then $p^2\mid A^2$, so that $p^2\mid 4$. Hence, $p=2$ and $2\mid A$. Thus, since $B\equiv 1 \pmod{4}$, we must have that  $(A \mmod{4},B\mmod{4})=(0,1)$ from condition \eqref{m=3:C2}. But then $A-2\equiv 2 \pmod{4}$, contradicting the assumption that $2^2\mid (A-2)$. Therefore, $p=3$ and $p\nmid A$. Since $3^2\mid(A-2)$, it follows that $A\equiv 2\pmod{9}$, which implies that $(A\mmod{9},B\mmod{9})=(2,1)$, contradicting condition \eqref{m=3:C4} of Theorem \ref{Thm:m=3}. Hence, $A-2$ is squarefree. The proof that $A+2$ is squarefree is similar, and so we omit the details. Thus, $f(x)\in \FF_3$, and the proof for $B=1$ is complete. 
        
        Suppose now that $B=-1$, so that $f(x)=x^6+Ax^3-1$. If $4\mid A$, then $(A \mmod{4},B \mmod{4})=(0,3)$ which contradicts condition \eqref{m=3:C2} of Theorem \ref{Thm:m=3}. Thus, $4\nmid A$. Similarly, if $9\mid A$, then $(A \mmod{9},B \mmod{9})=(0,8)$ which contradicts condition \eqref{m=3:C2} of Theorem \ref{Thm:m=3}; and if $3\nmid A$ with $(A \mmod{9},B \mmod{9})\in \{(4,8),(5,8)\}$, we have a contradiction with \eqref{m=3:C4} of Theorem \ref{Thm:m=3}. Consequently, $A \mmod{9}\not \in \{0,4,5\}$. Next, we claim that $\Omega:=(A^2+4)/\gcd(A^2+4,4)$ is squarefree. Note that 
         \[\Omega=\left\{\begin{array}{cl}
           (A^2+4)/4 & \mbox{if $2\mid A$},\\[.5em]
           A^2+4 &  \mbox{if $2\nmid A$.}
         \end{array} \right.\] By way of contradiction, assume that $p^2\mid \Omega$ for some prime $p$. 
         If $2\mid A$, then $2\mid \mid \Omega$ since $4\nmid A$. Thus, $p\ge 3$ regardless of the parity of $A$. Since $f(x)$ is monogenic, we have by condition \eqref{m=3:C5} of Theorem \ref{Thm:m=3} that $p\mid 3A$. Observe that $3\nmid \Omega$ since $-1$ is not a square modulo 3; and if $p\mid A$, we arrive at the contradiction that $p=2$. Hence, the claim that $\Omega$ is squarefree has been established, which verifies that $f(x)\in \FF_4$, and completes the proof of item \eqref{M3:I2} of the theorem.

 \subsection*{{\bf The Case} $\mathbf{C_3 \times S_3}$}
  \subsubsection*{$\mathbf{\FF_5}$} Let $f(x)=x^6+Ax^3+(A^2+3)/4$ such that
 \[A\mmod{2}=1,\  A\ne \pm 1,\ \mbox{with}\ B=(A^2+3)/4 \ \mbox{squarefree.}\] Let $h(x)=x^2+Ax+(A^2+3)/4$, which is irreducible over $\Q$ since $\Delta(h)=-3$. Then, if $f(x)=h(x^3)$ is reducible, it follows from Theorem \ref{Thm:m=3 HJMS} that the Diophantine equation
 \begin{equation}\label{Eq:Dio}
 \D:\quad  (A^2+3)/4=n^3
 \end{equation} has a solution. Hence, $n=\pm 1$ since $(A^2+3)/4$ is squarefree. If $n=-1$, then
 $\D$ in \eqref{Eq:Dio} has no integer solutions, and if $n=1$, we see that $A=\pm 1$, which contradicts the restriction on $A$.   Hence, $f(x)$ is irreducible over $\Q$.

 From \eqref{Eq:Delf}, we have
 \begin{equation*}\label{Delf F5}
 \Delta(f)=-3^9\left(\frac{A^2+3}{4}\right)^2.
 \end{equation*} For the monogenicity of $f(x)$, we consider first the prime divisor $p=3$ of $\Delta(f)$. If $3\mid A$, then $3\mid (A^2+3)/4$. Thus, condition \eqref{m=3:C1} of Theorem \ref{Thm:m=3} is satisfied since $(A^2+3)/4$ is squarefree. If $3\nmid A$, then $3\nmid (A^2+3)/4$ and, with the restrictions on $A$, we have
 \[(A \mmod{9},\ B \mmod{9})\in \{(1, 1), (2, 4), (4, 7), (5, 7), (7, 4), (8, 1)\}.\]
  Thus, condition \eqref{m=3:C4} of Theorem \ref{Thm:m=3} is satisfied.

  Suppose next that $p\ne 3$ is a prime divisor of $(A^2+3)/4$. Observe that $p\ne 2$ since $2\nmid A$, and that $p\nmid A$. Thus, condition \eqref{m=3:C3} of Theorem \ref{Thm:m=3} is satisfied since $(A^2+3)/4$ is squarefree. Hence, $f(x)$ is monogenic.

  To show that $\Gal(f)\simeq C_3 \times S_3$, we use Theorem \ref{Thm:HJCSextic1}. We have already shown that the equation $\D$ in \eqref{Eq:Dio} has no integer solutions. An easy calculation reveals that $-3\delta=3^2$. Finally, we must show that
  \[R(x)=x^3-3\left(\frac{A^2+3}{4}\right)x+A\left(\frac{A^2+3}{4}\right)\] is irreducible over $\Q$. From the previous discussion, we see that
  \[d:=\gcd(A,(A^2+3)/4)\in \{1,3\}.\] If $d=1$, then there exists a prime divisor $p\ne 3$ of $(A^2+3)/4$ such that $p\nmid A$. Since $(A^2+3)/4$ is squarefree, it follows that $R(x)$ is $p$-Eisenstein and, therefore, irreducible over $\Q$. If $d=3$, then we can let $A=3k$ for some $k\in \Z$. Then
  \[A\left(\frac{A^2+3}{4}\right)=3^2k\left(\frac{3k^2+1}{4}\right).\] If $\abs{k}>1$, then $(3k^2+1)/4>1$ and there exists a prime divisor $p$ of $(3k^2+1)/4$ such that $p\nmid 3^2k$. Since $(3k^2+1)/4$ is squarefree, it follows that $R(x)$ is $p$-Eisenstein and irreducible over $\Q$. If $\abs{k}=1$, then
  \[R(x)=x^3-9x-9 \quad \mbox{or} \quad R(x)=x^3-9x+9,\] both of which are irreducible in $\F_2[x]$, and hence irreducible over $\Q$. Thus, $\Gal(f)\simeq C_3\times S_3$.

  Using arguments similar to previous situations show that trinomials in $\FF_5$ with positive coefficient on $x^3$ are all distinct, and that there exist infinitely many such trinomials in $\FF_5$.

For the converse, we assume that $f(x)=x^6+Ax^3+B$ is monogenic with $\Gal(f)\simeq C_3 \times S_3$, and we show that $f(x)\in \FF_5$. We first note that $2\nmid A$, and since the proof is identical to the proof given in the last two paragraphs of {\bf The Case} $\mathbf{S_3}$, we omit the details here. Then, $2\nmid (A^2-4B)$ so that
\[\Gamma=\dfrac{A^2-4B}{3^{\nu_3(A^2-4B)}}.\] Since $\Gamma$ is squarefree, and $-3(A^2-4B)$ is a square  by Theorem \ref{Thm:HJCSextic1}, it follows that 
\begin{equation*}
  -3(A^2-4B)=3^{2b}
\end{equation*} for some integer $b\ge 1$,
which implies that 
\begin{equation}\label{Eq:B}
  B=\dfrac{A^2+3^{2b-1}}{4}.
\end{equation} 

Next, we claim that $A\ne \pm 1$. Assume, by way of contradiction, that $A=\pm 1$. If $B=1$, then $f(x)\in \FF_1$, and if $B=-1$, then $f(x)\in \FF_4$. Thus, $\abs{B}\ge 2$.  Observe that if $b=1$ in \eqref{Eq:B}, then $B=1$, which contradicts the fact that $\abs{B}\ge 2$. Hence, $b\ge 2$, so that $B\equiv 7 \pmod{9}$. Therefore, $(A \mmod{9}, B \mmod{9})\in \{(1,7), (8,7)\}$, which contradicts condition \eqref{m=3:C4} of Theorem \ref{Thm:m=3}. Thus, $A\ne \pm 1$.  

Suppose that $b\ge 2$ in \eqref{Eq:B}. If $3\mid A$, then $9\mid B$, which contradicts the fact that $B$ is squarefree. Hence, $3\nmid A$ and $B\equiv 7A^2 \pmod{9}$ from \eqref{Eq:B}, so that $3\nmid B$. Consequently,  
\[(A \mmod{9}, \ B\mmod{9})\in \{(1,7),(2,1),(4,4),(5,4),(7,1),(8,7)\},\] which contradicts condition \eqref{m=3:C4} of Theorem \ref{Thm:m=3}. Thus, $b=1$ and 
$B=(A^2+3)/4$ from \eqref{Eq:B}. Since $f(x)$ is monogenic, we have that $B$ is squarefree, which completes the proof of item \eqref{M3:I3} of the theorem.

  \subsection*{{\bf The Case} $\mathbf{S_3 \times S_3}$}
  Aside from $B$ and $\Gamma$ being squarefree, the main focus in Theorem \ref{Thm:m=3} is on primes $p\in \{2,3\}$ for the monogenicity of $f(x)$. An examination of Theorem \ref{Thm:m=3} reveals that   
  \begin{itemize}
    \item condition \eqref{m=3:C2} addresses the prime $p=2$ in the  cases $A \mmod{4}\in \{0,2\}$ when $2\mid A$, 
           while the prime $p=3$ is handled in 12 separate cases when $3\mid A$;
        \item condition \eqref{m=3:C3} addresses the prime $p=3$ in exactly 8 separate cases;
        \item condition \eqref{m=3:C4} addresses the prime $p=3$ in exactly 28 separate cases.   
  \end{itemize} Certainly, $p=2$ does not necessarily divide $\Delta(f)$, but $p=3$ always divides $\Delta(f)$. Thus, we must also consider the third possibility for the prime $p=2$ that $2\nmid A$, together with the possibilities for $p=3$ found in Theorem \ref{Thm:m=3}. This analysis leads to a total of 3(12+8+28)=144 mutually exclusive cases. In each of these cases, an infinite 2-parameter family of monogenic trinomials $f(x)$  can be constructed, and with suitable additional restrictions if necessary, there exist infinitely many trinomials in each of these families with $\Gal(f)\simeq S_3\times S_3$. We point out that such an approach could have been used for the other Galois groups, but it does not seem to lend itself so readily to yield single-parameter families.   
  
  Since the methods are similar in each of the $144$ cases, we provide details in only two cases. The first case is 
  \begin{equation}\label{Eq:first case}
  2\nmid A \quad \mbox{and} \quad 3\mid A \quad \mbox{with} \quad (A \mmod{9}, \ B \mmod{9})=(0,2).
  \end{equation} By the Chinese Remainder Theorem, we have that $A\equiv 9 \pmod{18}$, and so we can let $A=18s+9$, where $s$ is an integer parameter. Letting $B=9t+2$, where $t$ is an integer parameter, we then have that 
  \begin{equation}\label{Eq:case 1} 
  f(x)=x^6+Ax^3+B=x^6+(18s+9)x^3+9t+2,
  \end{equation} with 
  \[\Gamma=\delta=A^2-4B=324s^2+324s-36t+73\] and 
  \[\Delta(f)=3^6(9t+2)^2(324s^2+324s-36t+73)^3.\] There exist infinitely many integers $t$ such that $B=9t+2$ is squarefree \cite{Prachar}, and for each such value of $t$, there exist infinitely many integers $s$ such that $\Gamma$ is squarefree \cite{Nagel}. Let $s$ and $t$ be such integers with $\Gamma\ne 1$. Then $f(x)$ is irreducible over $\Q$ by Corollary \ref{Cor:m=3 HJMS}, and is therefore monogenic by construction. 
  
  Since $3\nmid (A^2-4B)$, we see that $-3\delta$ is not a square. If $B=9t+2=c^3$ for some $c\in \Z$, then $9t+2=\pm 1$ since $9t+2$ is squarefree, which yields the contradiction that $t\in \{-1/3,-1/9\}$.  
    Thus, $9t+2\ne c^3$ for any $c\in \Z$. Finally, if $2\nmid (9t+2)$, then $R(x)$ is irreducible in $\F_2[x]$, and if $2\mid (9t+2)$, then $R(x)$ is $2$-Eisenstein since $B$ is squarefree. Hence, $R(x)$ is irreducible over $\Q$ and $\Gal(f)\simeq S_3\times S_3$ by Theorem \ref{Thm:HJCSextic1}. We note that no additional restrictions on the parameters $s$ and $t$ are required in this case to achieve $\Gal(f)\simeq S_3\times S_3$.   
  
  Thus, with $f(x)$ from \eqref{Eq:case 1}, we have shown that 
  \[\FF_6:=\{f(x): B \ \mbox{and} \ \Gamma \ \mbox{are squarefree}\}\] is the infinite set of all monogenic trinomials $f(x)$ satisfying \eqref{Eq:first case} such that $\Gal(f)\simeq S_3\times S_3$. Let 
  \[\widehat{\FF_6}:=\{f(x)\in \FF_6: A>0 \ \mbox{and} \ B=2\}.\] Observe that since 
  \[A=18s+9=9(2s+1)>0,\] then $A\ge 9$. We claim next that $f_1(x), f_2(x)\in \widehat{\FF_6}$, where 
  \[f_1(x)=x^6+A_1x^3+2 \quad \mbox{and} \quad f_2(x)=x^6+A_2x^3+2\in \widehat{\FF_6},\] with $A_1\ne A_2$, are distinct. By way of contradiction, assume that $f_1(x)$ and $f_2(x)$ are not distinct. Then, by Corollary \ref{Cor:distinct}, we have that 
  $\Delta(f_1)=\Delta(f_2)$, or equivalently,
  \begin{equation}\label{Eq:F6}
  2^23^6(A_1+A_2)(A_1-A_2)(A_2^4+(A_1^2-24)A_2^2+A_1^2(A_1^2-24)+192)=0,
  \end{equation} since 
  \[\Delta(f_i)=2^23^6(A_i^2-8)^3.\] Since $A_i\ge 9$, it follows that 
  \[A_1+A_2>0 \quad \mbox{and} \quad A_2^4+(A_1^2-24)A_2^2+A_1^2(A_1^2-24)+192>0,\] which yields the contradiction that $A_1=A_2$ in \eqref{Eq:F6}. Thus, since there exist infinitely many positive integers $A$ such that $A^2-8$ is squarefree \cite{Nagel}, we have shown that $\FF_6$ contains infinitely many distinct monogenic trinomials. 
  
  The second case we consider is 
  \begin{align}\label{Eq:second case}
  \begin{split}
  2\mid A \quad &\mbox{and} \quad 3\nmid AB \quad \mbox{with}\\
   (A \mmod{4}, \ B \mmod{4})=(0,1) &
   \quad \mbox{and} \quad  (A \mmod{9}, \ B \mmod{9})=(1,1).
   \end{split}
  \end{align}  Then, by the Chinese Remainder Theorem, we have that 
  \[(A \mmod{36}, \ B \mmod{36})=(28,1).\] Hence, we can write 
  \[A=36s+28 \quad \mbox{and} \quad B=36t+1,\] where $s$ and $t$ are integer parameters. Thus,
  \begin{equation}\label{Eq:f Case 2} 
  f(x)=x^6+Ax^3+B=x^6+(36s+28)x^3+36t+1
  \end{equation} and
  \begin{align*}
     A^2-4B&=(36s+28)^2-4(36t+1)\\
     &=1296s^2+2016s-144t+780\\
     &=2^23(108s^2+168s-12t+65).
  \end{align*} Therefore, 
  \begin{align*}
    \Delta(f)&=3^6B^2(A^2-4B)^3\\
    &=3^6(36t+1)^2((36s+28)^2-4(36t+1))^3\\
    &=2^63^9(36t+1)^2(108s^2+168s-12t+65)^3
  \end{align*} 
  and 
  \begin{equation}\label{Eq:Gam case 2}
  \Gamma=108s^2+168s-12t+65.
   \end{equation} There exist infinitely many integers $t$ such that $36t+1$ is squarefree \cite{Prachar}, and for each such value of $t$, there exist infinitely many integers $s$ for which $\Gamma$ is squarefree \cite{Nagel}. Let $s$ and $t$ be such integers. Since $\Gamma$ is squarefree, then $A^2-4B$ is a square if and only if $\Gamma=3$, which is impossible since $\Gamma\equiv 2 \pmod{3}$ from \eqref{Eq:Gam case 2}. Thus, $g(x)=x^2+Ax+B$ is irreducible over $\Q$, so that $f(x)$ is irreducible over $\Q$ by Corollary \ref{Cor:m=3 HJMS}, and is therefore monogenic by construction. 
  
  If $-3\delta=-2^23^2\Gamma$ is a square, then $\Gamma=-1$ since $\Gamma$ is squarefree. Then, $\Gamma+1=0$ and $4\mid(\Gamma+1)$. However, we see from \eqref{Eq:Gam case 2} that $\Gamma+1\equiv 2 \pmod{4}$, and this contradiction establishes the fact that $-3\delta$ is not a square. 
  If $B=36t+1=c^3$ for some $c\in \Z$, then $36t+1=\pm 1$ since $B$ is squarefree. If $36t+1=-1$, then $t=-1/18$, which is impossible. However, if $36t+1=1$, then $t=0$, and we must exclude the value $t=0$ from the set of viable values of $t$. Finally, since $2\nmid B$, then $R(x)$ is irreducible in $\F_2[x]$, and therefore, $R(x)$ is irreducible over $\Q$. Hence, $\Gal(f)\simeq S_3\times S_3$ by Theorem \ref{Thm:HJCSextic1}. 
  
  Thus, with $f(x)$ from \eqref{Eq:f Case 2}, we have shown that 
  \[\FF_7:=\{f(x): t\ne 0; \ B \ \mbox{and} \ \Gamma \ \mbox{are squarefree}\}\] is the infinite set of all monogenic trinomials $f(x)$ satisfying \eqref{Eq:second case} such that $\Gal(f)\simeq S_3\times S_3$.
  Let 
  \[\widehat{\FF_7}:=\{f(x)\in \FF_7: A>0 \ \mbox{and} \ B=37\}.\] Observe that since 
  \[A=36s+28=4(9s+7)>0,\] then $A\ge 28$. We claim next that $f_1(x), f_2(x)\in \widehat{\FF_7}$, where 
  \[f_1(x)=x^6+A_1x^3+37 \quad \mbox{and} \quad f_2(x)=x^6+A_2x^3+37\in \widehat{\FF_7},\] with $A_1\ne A_2$, are distinct. By way of contradiction, assume that $f_1(x)$ and $f_2(x)$ are not distinct. Then, by Corollary \ref{Cor:distinct}, we have that 
  $\Delta(f_1)=\Delta(f_2)$, or equivalently,
  \begin{equation}\label{Eq:F7}
  3^637^2(A_1+A_2)(A_1-A_2)(A_2^4+(A_1^2-144)A_2^2+A_1^2(A_1^2-144)+65712)=0,
  \end{equation} since 
  \[\Delta(f_i)=3^637^2(A_i^2-148)^3.\] Since $A_i\ge 28$, it follows that 
  \[A_1+A_2>0 \quad \mbox{and} \quad A_2^4+(A_1^2-144)A_2^2+A_1^2(A_1^2-144)+65712>0,\] which yields the contradiction that $A_1=A_2$ in \eqref{Eq:F6}. Thus, since there exist infinitely many positive integers $A$ such that $A^2-148$ is squarefree \cite{Nagel}, we have shown that $\FF_7$ contains infinitely many distinct monogenic trinomials, and the proof of the theorem is complete. 
  \end{proof}







\end{document}